\documentclass[graybox]{svmult}

\usepackage{helvet}         
\usepackage{courier}        
\usepackage{type1cm}        
\usepackage{makeidx}         
\usepackage{graphicx}        
\usepackage{multicol}        
\usepackage[bottom]{footmisc}
\usepackage{bm}
\usepackage{latexsym}
\usepackage{dsfont}
\usepackage{bbm}
\usepackage{amssymb}
\usepackage{amsmath}
\usepackage{color}
\usepackage{soul}
\usepackage{mathtools}
\usepackage[top=3.cm, bottom=3cm, outer=3.6cm, inner=3.6cm]{geometry}

\newtheorem{Theorem}{Theorem}[section]
\newtheorem{Lemma}[Theorem]{Lemma}

\newtheorem{Prop}[Theorem]{Proposition}
\newtheorem{Rem}[Theorem]{Remark}

\def\Erw{\mathbb{E}}

\def\N{\mathbb{N}}
  
\def\Prob{\mathbb{P}} 

\def\R{\mathbb{R}}

\def\Var{\mathbb{V}{\rm ar}}

\def\sfs{\mathsf{s}}

\def\sfm{\mathsf{m}}

\def\eps{\varepsilon}

\def\1{\vec{1}}
\def\3{{\ss}}

\def\eqdist{\stackrel{d}{=}}
\def\iprob{\stackrel{\Prob}{\to}}

\def\Plim{\mathop{\Prob\text{\rm -lim}\,}}

\allowdisplaybreaks[4] 

\begin{document}

\title*{Functional limit theorems for the number of occupied boxes in the Bernoulli sieve}
\titlerunning{The number of occupied boxes in the Bernoulli sieve}
\author{Gerold Alsmeyer$^{1}$, Alexander Iksanov$^{2}$ and Alexander Marynych$^{1,2}$}
\institute{$^{1}$ Institute of Mathematical Statistics, Department
of Mathematics and Computer Science, University of M\"unster,
Orl\'eans-Ring 10, D-48149 M\"unster, Germany.\at
$^{2}$ Faculty of Cybernetics, Taras Shevchenko National University of Kyiv, 01601 Kyiv, Ukraine\at
\email{gerolda@math.uni-muenster.de, iksan@univ.kiev.ua, marynych@unicyb.kiev.ua}}

\maketitle

{\abstract{\small The Bernoulli sieve is the infinite Karlin ``balls-in-boxes'' scheme with random probabilities of stick-breaking type. Assuming that the number of placed balls equals $n$, we prove several functional limit theorems (FLTs) in the Skorohod space $D[0,1]$ endowed with the $J_{1}$- or $M_{1}$-topology for the number $K_{n}^{*}(t)$ of boxes containing at most $[n^{t}]$ balls, $t\in[0,1]$, and the random distribution function $K_{n}^{*}(t)/K_{n}^{*}(1)$, as $n\to\infty$. The limit processes for $K_{n}^{*}(t)$ are of the form $(X(1)-X((1-t)-))_{t\in[0,1]}$, where $X$ is either a Brownian motion, a spectrally negative stable L\'evy process, or an inverse stable subordinator. The small values probabilities for the stick-breaking factor determine which of the alternatives occurs. If the logarithm of this factor is integrable, the limit process for $K_{n}^{*}(t)/K_{n}^{*}(1)$ is a L\'{e}vy bridge. Our approach relies upon two novel ingredients and particularly enables us to dispense with a Poissonization-de-Poissonization step which has been an essential component in all the previous studies of $K_{n}^{*}(1)$. First, for any Karlin occupancy scheme with deterministic probabilities $(p_{k})_{k\ge 1}$, we obtain an approximation, uniformly in $t\in[0,1]$, of the number of boxes with at most $[n^{t}]$ balls by a counting function defined in terms of $(p_{k})_{k\ge 1}$. Second, we prove several FLTs for the number of visits to the interval $[0,nt]$ by a perturbed random walk, as $n\to\infty$. If the stick-breaking factor has a beta distribution with parameters $\theta>0$ and $1$, the process $(K_{n}^{*}(t))_{t\in[0,1]}$ has the same distribution as a similar process defined by the number of cycles of length at most $[n^{t}]$ in a $\theta$-biased random permutation a.k.a.\ a Ewens permutation with parameter $\theta$. As a consequence, our FLT with Brownian limit forms a generalization of a FLT obtained earlier in the context of Ewens permutations by DeLaurentis and Pittel \cite{DeLaurentisPittel:85}, Hansen \cite{Hansen:90}, Donnelly, Kurtz and Tavar\'e \cite{DonKurTav:91}, and Arratia and Tavar\'e \cite{ArratiaTav:92}.}}

\bigskip



\section{Introduction}\label{sec:intro}

\subsection{The Bernoulli sieve and regenerative compositions}
Given a sequence $W_{1},W_{2},\ldots$ of independent copies of a $(0,1)$-valued random variable $W$, consider the random partition of $(0,1]$ into the subintervals $(V_{i},V_{i-1}]$, $i\in\N$, called boxes hereafter, where
$$ V_{0}\,:=\,1,\quad\text{and}\quad V_{n}\,:=\,\prod_{i=1}^{n}W_{i}\quad\text{for }n\in\N. $$
One may also think of the $V_{i}$ as the successive cut points in a stick-breaking procedure at the end of which a stick with endpoints 0 and 1 is broken into the infinitely many pieces with endpoints $V_{i}$ and $V_{i-1}$, $i\in\N$. Next, let $(U_{j})_{j\in\N}$ be a sequence of independent uniform $(0,1)$ random variables which is also independent of $(W_{k})_{k\in\N}$. The \emph{Bernoulli sieve} is a random occupancy scheme in which balls, labeled by $1,2,\ldots$ and of random weights $U_{1},U_{2},\ldots$, are placed into the boxes in accordance with their weight. Hence, a ball of weight $U$ goes into the box $(V_{i},V_{i-1}]$ iff $V_{i}<U\le V_{i-1}$. Since its introduction by Gnedin \cite{Gnedin:04}, the model has been studied in a series of articles \cite{GneIksMar:10b,GneIksMar:10,GneIksMar:12,GneIksRos:08,GneIksNegRos:09,Iksanov:12,Iksanov:13,IksMarVat:15}. The term ``Bernoulli sieve'' can be understood when interpreting the occupancy scheme in terms of a randomized leader-election procedure (see \cite{Gnedin:04} for more details).

Defining the (infinite) vector of occupancy counts
\begin{equation}\label{zni_def}
Z^{*}_{n,i}\ :=\ \#\{1\le j\le n: U_{j}\in (V_{i},\,V_{i-1}]\},\quad i\in\N,
\end{equation}
thus $Z^{*}_{n,i}\ge 0$ for $i\in\N$ and $\sum_{i\ge 1}Z^{*}_{n,i}=n$, we see that the $\mathcal{Z}^{*}_{n}:=(Z^{*}_{n,i})_{i\in\N}$ induces a \emph{weak composition} of the integer $n$. The attribute ``weak'' is intended to emphasize that zero parts of the composition are allowed. 
By sequentially allocating the points $U_{1},U_2,\ldots$, the sequence $(\mathcal{Z}^{*}_{n})_{n\ge 0}$ is consistently defined for all $n\in\N_{0}$, and it further has the following two distinguished regenerative properties:
\begin{itemize}\itemsep2pt
\item \emph{Sampling consistency}: If one out of $n$ points is removed uniformly at random from the interval it belongs to, then the resulting weak composition of $n-1$ has the same law as $\mathcal{Z}^{*}_{n-1}$.
\item \emph{The deletion property}: If the first interval $(V_{1},1]$ contains $m$ points and is removed, then a weak composition of $n-m$ with the same law as $\mathcal{Z}^{*}_{n-m}$ is obtained.
\end{itemize}

The class of random compositions generated by a Bernoulli sieve does not cover all regenerative compositions (i.e., those having the two aforementioned properties). In fact,
Theorem 5.2 in \cite{GnedinPitman:05} states that every consistent family of regenerative compositions can be constructed by allocating the points $U_{1},U_{2},\ldots$ to countably many open intervals forming the complement of the closed range of a multiplicative zero-drift subordinator $(e^{-L_{t}})_{t\ge 0}$ independent of the uniform sample. Within this more general framework, weak compositions pertaining to a Bernoulli sieve are those corresponding to a compound Poisson process $(L_{t})_{t\ge 0}$.

\vspace{.1cm}
In the classical occupancy scheme of Karlin \cite{Karlin:67} balls are placed independently in an infinite array of boxes in accordance with a probability vector $(p_{k})_{k\in\N}$, where $p_{k}$ denotes the probability of choosing box $k$. The Bernoulli sieve is the Karlin occupancy scheme with random probability
\begin{equation}\label{BS_frequencies}
p^{*}_{k}\ =\ V_{k-1}-V_{k}\ =\ W_{1}\cdots W_{k-1}(1-W_{k})
\end{equation}
of choosing box $k\in\N$ and such that, given $(p_{k}^{*})_{k\in\N}$, balls are allocated independently. The Bernoulli sieve can therefore be thought of as an \emph{occupancy scheme in random environment}, the latter being defined by the i.i.d. random variables $W_{1},W_{2},\ldots$.

For $r=1,\ldots, n$, denote by $K^{*}_{n,r}:=\sum_{i\ge 1}\1_{\{Z^{*}_{n,i}=r\}}$ the number of boxes containing exactly $r$ balls, and let $K^{*}_{n}:=\sum_{r=1}^n K^{*}_{n,r}=\sum_{i\ge 1}\1_{\{Z^{*}_{n,i}\ge 1\}}$ denote the number of nonempty boxes. Then define
\begin{equation*}\label{knt_def}
K^{*}_{n}(t)\ :=\ \sum_{r=1}^{[n^{t}]}K^{*}_{n,r}\ =\ \sum_{i \ge 1}\1_{\{Z^{*}_{n,i}\in [1,n^{t}]\}}
\end{equation*}
for $t\in [0,1]$.

\vspace{.1cm}
Throughout the rest of the paper, we will use the following notational rule. Quantities related to the Bernoulli sieve are starred, whereas corresponding quantities for the general Karlin scheme are not. The same rule applies, for the most part, to perturbed random walks to be defined in Section \ref{sec_prw}.

\subsection{$J_{1}$- and $M_{1}$-topology: a brief review}

For $T>0$, let $D[0,T]$ denote the Skorohod space of real-valued functions on $[0,T]$,
which are right-continuous with left-hand limits. We will need the $J_{1}$- and the
$M_{1}$-topology on $D[0,T]$ which are commonly used and were introduced in a famous paper by Skorohod \cite{Skorohod:56}. The $J_{1}$-topology is generated by the metric
$$ d(f,g)\ :=\ \inf_{\lambda\in\Lambda}\left(\sup_{y\in[0,T]}|f(\lambda(t))-g(t)|\vee
\sup_{t\in[0,T]}|\lambda(t)-t|\right), $$
where $\Lambda$ denotes the class of strictly increasing, continuous functions $\lambda:[0,T]\to [0,T]$ with $\lambda(0)=0$ and $\lambda(T)=T$. Functions $f_{n}$ which are $J_{1}$-convergent to a limit function $f$ are allowed to have a single jump in the vicinity of a jump of $f$. Furthermore, the positions of the jumps of $f_{n}$ and their magnitudes should converge to the positions of the jumps of $f$ and their magnitude. This is in contrast to locally uniform convergence which requires the positions of jumps of $f_{n}$ and $f$ to be the same rather than asymptotically equal.

The $M_{1}$-topology is weaker than the $J_{1}$-topology and $M_{1}$-convergence of $f_{n}$ to $f$ is equivalent to the convergence of the closed graph of $f_{n}$ to the closed graph of $f$. For instance, choosing $f_{n}(t):=\1_{[1-1/n,\,1+1/n)}(t)+2\cdot\1_{[1+1/n,\,2]}(t)$ and $f(t):=2\cdot\1_{[1,\,2]}(t)$, the $f_{n}$ do converge to $f$ in the $M_{1}$-topology, but not in the $J_{1}$-topology on $D[0,2]$. Without going into details, we mention that the $M_{1}$-topology is typically used in functional limit theorems in which the limit process has jumps unmatched in the convergent sequence of processes. This may happen, for example, if the converging processes are a.s.\ continuous or have asymptotically vanishing jumps, while the limit process is discontinuous with positive probability. We refer the reader to the monograph \cite{Whitt:02} for a comprehensive exposition of the $J_{1}$- and the $M_{1}$-topologies as well as some
other topologies on $D[0,T]$.

Throughout the paper $\overset{J_{1}}{\Longrightarrow}$ and $\overset{M_{1}}{\Longrightarrow}$ will mean weak convergence in the Skorohod space when endowed with the $J_{1}$-topology and the $M_{1}$-topology, respectively. Furthermore, we will use $\iprob$ and also $\Plim$ to denote convergence in probability with respect to $\Prob$. Finally, $\eqdist$ will stand for equality in distribution.

\subsection{Ewens permutations and Ewens sampling formula}

Let $\mathfrak{S}_{n}$ be the symmetric group of order $n$. The Ewens family of random permutations is a parametric family $\Pi_{n}:=\Pi_{n}(\theta)$, $\theta>0$, of random objects taking values in $\mathfrak{S}_{n}$ with probabilities
$$ \Prob\{\Pi_{n}=\sigma\}\ =\ \frac{\Gamma(\theta)\theta^{|\sigma|}}{\Gamma(n+\theta)},\quad
\sigma\in\mathfrak{S}_{n}, $$ 
where $|\sigma|$ denotes the number of cycles in $\sigma$ and $\Gamma$ is the Euler gamma function. Plainly, $\Pi_{n}(1)$ is a uniform random permutation of $\{1,\ldots,n\}$ for which all $n!$ permutations are equally likely.

For $r=1,\ldots, n$, denote by $C_{n,r}$ the number of cycles of
length $r$ in $\Pi_{n}$. The following is the famous Ewens sampling formula:
$$ \Prob\{C_{n,1}=c_{1},\ldots,C_{n,n}=c_{n}\}\ =\ \frac{n!\Gamma(\theta)}{\Gamma(\theta+n)}\prod_{i=1}^{n}\frac{\theta^{c_{i}}}{i^{c_{i}}c_{i}!}\,\1_{\{\sum_{i=1}^{n}ic_{i}=n\}}. $$
Define $C_{n}(t):=\sum_{r=1}^{[n^{t}]}C_{n,r}$ for $t\in [0,1]$. A remarkable result, originally due to DeLaurentis and Pittel \cite{DeLaurentisPittel:85} for the uniform case
$\theta=1$ and to Hansen \cite{Hansen:90} for the general case $\theta>0$, asserts that
\begin{equation}\label{ESF_FLT}
\left(\frac{C_{n}(t)-\theta t\log n}{\sqrt{\theta\log n}}\right)_{t\in [0,1]}\quad\overset{J_{1}}{\Longrightarrow}\quad (B(t))_{t\in [0,1]},\quad n\to\infty,
\end{equation}
where $(B(t))_{t\in[0,1]}$ denotes a standard Brownian motion. Later, much simpler proofs of \eqref{ESF_FLT} were found by Donnelly, Kurtz and Tavar\'{e} \cite{DonKurTav:91} and Arratia and Tavar\'e \cite{ArratiaTav:92}. While the first work is based on a Poisson embedding, the second one uses Feller coupling \cite[p.~16]{ArrBarTav:03} as a key tool.

\vspace{.1cm}
The connection between the Ewens permutations and the Bernoulli sieve emerges when choosing $W$ to have a beta distribution with parameters $\theta>0$ and $1$, i.e. $\Prob\{W\in {\rm d}x\}=\theta x^{\theta-1}\1_{(0,1)}(x){\rm d}x$. In this case (see, for
instance, Example 2 in \cite{GnedinPitman:05} or Section 5.4 in
\cite{ArrBarTav:03})
\begin{equation*}
(K^{*}_{n,1},\ldots,K^{*}_{n,n})\ \eqdist\ (C_{n,1},\ldots,C_{n,n}),\quad n\in\N
\end{equation*}
which implies that
\begin{equation}\label{Ewens_as_BS}
(K^{*}_{n}(t))_{t\in[0,1]}\ \eqdist\ (C_{n}(t))_{t\in[0,1]}.
\end{equation}

\section{Main results and discussion}\label{main}

The asymptotic behavior of the small-parts counts in the Bernoulli sieve is well-understood if $\Erw(|\log W|)<\infty$. According to Theorem 3.3 in \cite{GneIksRos:08},
the vector $(K^{*}_{n,1},\ldots, K^{*}_{n,j})$, with $j$ fixed, converges in distribution to a similar vector defined in terms of a limiting ``balls-in-boxes'' scheme in which ball weights are identified with the arrival times of a standard Poisson process on $[0,\infty)$ and boxes are formed by successive points of $\exp(A)$ for a stationary renewal point process $A$ on $\R$ driven by the distribution of $|\log W|$. A criterion for weak convergence of $K^{*}_{n}$ is given in Theorem 1.1 and Corollary 1.1 of \cite{GneIksMar:10}. One consequence of these results is that, if $\Erw(|\log W|)<\infty$, the contribution of the small-parts counts to $K^{*}_{n}$ becomes asymptotically negligible, as $n\to\infty$, and it raises the natural question which components of the vector $(K^{*}_{n,1},\ldots, K^{*}_{n,n})$ provide the main contribution to $K^{*}_{n}$ for large $n$. If $\Erw|\log W|<\infty$, the answer is provided by Proposition \ref{probab}, for the case $\Erw|\log W|=\infty$ see \eqref{infinite}.

\begin{Prop}\label{probab}
If $\mu:=\Erw|\log W|<\infty$, then
\begin{equation}\label{df}
\Plim_{n\to\infty}\sup_{t\in [0,1]}\,\bigg|\frac{K^{*}_{n}(t)}{K^{*}_{n}}-t\bigg|\ =\ 0.
\end{equation}
\end{Prop}

Thus, if $\mu<\infty$, the random distribution function $t\mapsto K^{*}_{n}(t)/K^{*}_{n}$ converges uniformly in probability to the uniform distribution function. This provides a definite answer to the question above, namely, for each $t<s$, $t,s\in [0,1]$, the asymptotic contribution (in probability) of the vector $(K^{*}_{n, [n^t]},\ldots,
K^{*}_{n,[n^s]})$ to $K^{*}_{n}$ equals $s-t$.


\vspace{.1cm}
Let us compare this observation with some results of a similar flavor from the literature and  point out beforehand that $\rho^{*}(x)$, defined by
\begin{equation}\label{rho_def_BS}
\rho^{*}(x)\ :=\ \#\{k\in\N: p^{*}_{k}\ge 1/x\}=\#\{k\in\N:W_{1}\cdot\ldots\cdot W_{k-1}(1-W_{k})\ge 1/x\}
\end{equation}
for $x>0$, exhibits a logarithmic growth (see Proposition \ref{strong} below). Consider now the Karlin occupancy scheme with deterministic or random $p_{k}$ such that $\rho(x)$, defined by
\begin{equation}\label{rho_def}
\rho(x)\ :=\ \#\{k\in\N:p_{k}\ge 1/x\},\quad x>0
\end{equation}
is regularly varying at infinity of index $\alpha$, $\alpha\in(0,1)$. Let $K_{n}$ and $K_{n,r}$ denote the number of occupied boxes and the number of boxes containing exactly $r$ balls, respectively, after $n$ balls have been placed. Then $\lim_{n\to\infty}K_{n,r}/K_{n}=c(r)$ a.s.\ for explicitly known constants $c(r)>0$, see Theorems 8 and 9 in \cite{Karlin:67}, Theorem 2.1 in \cite{GnePitYor:06} and Corollary 21 in \cite{GneHanPit:07} (interesting extensions can be
found in \cite{Schweinsberg:10}). Thus, in sharp contrast to \eqref{df}, the major contribution to $K_{n}$ is made by the small-parts counts.

\vspace{.1cm}
In view of Proposition \ref{probab}, it is natural to ask how fast $K^{*}_{n}(t)/K^{*}_{n}$ approaches uniformity in $D[0,1]$. We will answer this question by first proving a FLT for
the process $(K^{*}_{n}(t))_{t\in [0,1]}$, properly centered and normalized, and then make use of the continuous mapping theorem. Our main results, Theorem \ref{main_theorem_fclt} and Theorem \ref{main2_theorem_fclt} treat the case of finite and infinite $\mu$, respectively. 

\begin{Theorem}\label{main_theorem_fclt}
Assume that 
\begin{equation}\label{finite_moment_{n}u}
\Erw|\log(1-W)|^{a}\,<\,\infty
\end{equation}
for some $a>0$. Set
\begin{align*}
&\hspace{1.5cm}u_{n}(t)\ :=\ \mu^{-1}\int_{(1-t)\log n}^{\log n}\Prob\{|\log(1-W)|\le s\}\ {\rm d}s
\shortintertext{and}
&v_{n}(t)\ :=\ \mu^{-1}\int_{(1-t)\log n}^{\log n}\Prob\{|\log(1-W)|> s\}\ {\rm d}s\ =\ \mu^{-1}t\log n - u_{n}(t)
\end{align*}
for $t\in [0,1]$, where $\mu=\Erw|\log W|$.
\begin{description}[(C.3)]\itemsep3pt
\item[(A1)] If $\sigma^2:=\Var|\log W|<\infty$, then
\begin{equation*}\label{clt_finite_variance}
\left(\frac{K^{*}_{n}(t)-u_{n}(t)}{\sqrt{\mu^{-3}\sigma^2\log n}}\right)_{t\in[0,1]}\ \overset{J_{1}}{\Longrightarrow}\  (B(t))_{t\in[0,1]}
\end{equation*}
as $n\to\infty$, and
$$ \left(\sqrt{\mu\sigma^{-2}\log n}\left(\frac{K^{*}_{n}(t)}{K^{*}_{n}}-t+\frac{v_{n}(t)-tv_{n}(1)}{u_{n}(1)}\right)\!\right)_{t\in[0,1]}\ \overset{J_{1}}{\Longrightarrow}\ (B(t)-tB(1))_{t\in[0,1]}, $$
where $(B(t))_{t\in [0,1]}$ is a standard Brownian motion.
\item[(A2)] If $\sigma^2=\infty$ and
\begin{equation*}\label{normalization_sv_sec_moment}
\Erw(\log W)^{2}\1_{\{|\log W|\le x\}}\ \sim \ \ell(x),\quad x\to\infty,
\end{equation*}
for some $\ell$ slowly varying at infinity, then
\begin{equation*}\label{clt_sv_sec_moment}
\left(\frac{K^{*}_{n}(t)-u_{n}(t)}{\mu^{-3/2}c(\log n)}\right)_{t\in[0,1]}\ \overset{J_{1}}{\Longrightarrow}\ (B(t))_{t\in[0,1]}
\end{equation*}
as $n\to\infty$, and
$$ \left(\frac{\sqrt{\mu}\log n}{c(\log n)}\left(\frac{K^{*}_{n}(t)}{K^{*}_{n}}-t+\frac{v_{n}(t)-tv_{n}(1)}{u_{n}(1)}\right)\!\right)_{t\in[0,1]}\ \overset{J_{1}}{\Longrightarrow}\ (B(t)-tB(1))_{t\in[0,1]}. $$
where 
$c$ is a positive function satisfying
$\lim_{x\to\infty} c(x)^{-2} x\ell(c(x))=1$.
\item[(A3)] If
\begin{equation}\label{normalization_rv_sec_moment}
\Prob\{|\log W|>x\}\ \sim\ x^{-\alpha}\ell(x),\quad x\to\infty,
\end{equation}
for some $\alpha\in (1,2)$ and some $\ell$ slowly varying at
infinity, then
\begin{equation*}\label{clt_rv_sec_moment}
\left(\frac{K^{*}_{n}(t)-u_{n}(t)}{\mu^{-(\alpha+1)/\alpha}c(\log
n)}\right)_{t\in [0,1]}\ \overset{M_{1}}{\Longrightarrow}\ 
(S_{\alpha}(t))_{t\in[0,1]}
\end{equation*}
as $n\to\infty$, and
$$ \left(\frac{\mu^{1/\alpha}\log n}{c(\log
n)}\left(\frac{K^{*}_{n}(t)}{K^{*}_{n}}-t+\frac{v_{n}(t)-tv_{n}(1)}{u_{n}(1)}\right)\right)_{t\in[0,1]}\ \overset{M_{1}}{\Longrightarrow}\ 
(S_{\alpha}(t)-tS_{\alpha}(1))_{t\in[0,1]},
$$
where $c$ is a positive function satisfying
$\lim_{x\to\infty} c(x)^{-\alpha} x\ell(c(x))=1$ and
$(S_{\alpha}(t))_{t\in [0,1]}$ is a spectrally negative $\alpha$-stable L\'{e}vy process such that $S_{\alpha}(1)$ has characteristic function
\begin{equation}\label{cf_st}
u\ \mapsto\ \exp\{-|u|^\alpha
\Gamma(1-\alpha)(\cos(\pi\alpha/2)+i\sin(\pi\alpha/2)\, {\rm
sgn}(u))\}, \ u\in\R,
\end{equation}
with $\Gamma$ being the gamma function.
\end{description}
\end{Theorem}

\begin{Rem}\rm
Along similar lines as in \cite{GneIksMar:10} where weak convergence of $K^{*}_{n}$ is proved, it can be checked that moment condition \eqref{finite_moment_{n}u} is not needed to ensure weak convergence of the \emph{finite-dimensional distributions} of $(K^{*}_{n}(t))_{t\in[0,1]}$. Our proof of Theorem \ref{main_theorem_fclt} is based on the decomposition
\begin{align}
K^{*}_{n}(t)-u_{n}(t)\ &=\ \big(K^{*}_{n}(t)-\big(\rho^{*}(n)-\rho^{*}\big(n^{(1-t)-}\big)\big)\big)\notag\\
&\hspace{1.5cm}+\ \big(\rho^{*}(n)-\rho^{*}\big(n^{(1-t)-}\big)-u_{n}(t)\big)\label{repr}
\end{align}
with $\rho^{*}(x)$ as defined in \eqref{rho_def_BS}. It will be shown that, irrespective of \eqref{finite_moment_{n}u}, the first term on the right-hand side of \eqref{repr}, properly normalized, converges to zero uniformly in
probability. However, for dealing with the second term, we need condition \eqref{finite_moment_{n}u} (see the proof of Theorem \ref{prw_flt} below) but do not know whether it is really necessary.
\end{Rem}

\begin{Rem}\rm
Whenever the second-order term $v_{n}(t)$ of the centering of $K_{n}^{*}(t)$ is killed by
the normalizing constants which is the case, for instance, if $\Erw|\log (1-W)|<\infty$, the ``true'' centering for $K_{n}^{*}(t)$ is $\mu^{-1} t\log n$. Similarly, the term $\frac{v_{n}(t)-tv_{n}(1)}{u_{n}(1)}$ may then be omitted in the limit theorems for $K^{*}_{n}(t)/K^{*}_{n}$ because it vanishes asymptotically when multiplied by the corresponding normalization.

Assume now that $W$ has a beta distribution with parameters $\theta>0$ and $1$, giving
$$ \Erw|\log W|\,=\,\theta^{-1},\quad\Var|\log W|\,=\,\theta^{-2}\quad\text{and}\quad\Erw|\log(1-W)|^{a}\,<\,\infty\quad\text{for all }a>0. $$
Then part (A1) of Theorem \ref{main_theorem_fclt} together with the preceding remark yields
$$ \bigg(\frac{K^{*}_{n}(t)-\theta t\log n}{\sqrt{\theta\log n}}\bigg)_{t\in [0,1]}\ \overset{J_{1}}{\Longrightarrow}\ (B(t))_{t\in [0,1]},\quad n\to\infty, $$ 
and in view of \eqref{Ewens_as_BS}, this limit relation is equivalent to
\eqref{ESF_FLT}. Thus, we found yet another proof of \eqref{ESF_FLT}. In fact, our Theorem \ref{main_theorem_fclt} constitutes a generalization of \eqref{ESF_FLT}.

\vspace{.1cm}
As shown in \cite{GneIksMar:12} and \cite{GneIksRos:08}, respectively, similar generalizations exist for the Erd\"os-T\'uran law for the order of Ewens permutations \cite[Theorem 5.15 on p.~116]{ArrBarTav:03} and for the weak laws for small cycles in Ewens permutations \cite[Theorem 5.1 on p.~96]{ArrBarTav:03}. On the other hand, the independence-based tools used in the proofs related to these permutations \cite[Section 5]{ArrBarTav:03} are no longer available in the more general framework of the Bernoulli sieve and must therefore be replaced by methods of advanced renewal theory.
\end{Rem}

\begin{Theorem}\label{main2_theorem_fclt}
If relation \eqref{normalization_rv_sec_moment} holds with $\alpha\in (0,1)$, then
\begin{equation}\label{clt_{i}nf_exp}
\bigg(\frac{\ell(\log n)K^{*}_{n}(t)}{(\log n)^{\alpha}}\bigg)_{t\in[0,1]}\ \overset{J_{1}}{\Longrightarrow}\ \big(W^{\leftarrow}_{\alpha}(1)-W^{\leftarrow}_{\alpha}((1-t)-)\big)_{t\in[0,1]}
\end{equation}
as $n\to\infty$, and
\begin{equation}\label{infinite}
\bigg({K^{*}_{n}(t)\over
K^{*}_{n}}\bigg)_{t\in[0,1]}\ \overset{J_{1}}{\Longrightarrow}\ 
\bigg(1-{W^{\leftarrow}_{\alpha}((1-t)-)\over
W_\alpha^\leftarrow(1)}\bigg)_{t\in[0,1]},
\end{equation}
where $W^{\leftarrow}_{\alpha}(t):=\inf\{s\ge 0:W_{\alpha}(s)>t\}$ for $t\ge 0$ and $(W_{\alpha}(s))_{s\ge 0}$ is an $\alpha$-stable subordinator (nondecreasing L\'{e}vy
process) with Laplace exponent $-\log \Erw(-zW_{\alpha}(1))=\Gamma(1-\alpha)z^{\alpha}$, $z\ge 0$.
\end{Theorem}

In the remaining part of this section, we briefly describe our approach and the organization of the paper.

\vspace{.1cm}
The following heuristic sheds some light on the asymptotic behavior of the number $K_{n}$ of occupied boxes in the Karlin scheme. Given $(p_{j})_{j\in\N}$, call a box $k$ large if $p_{k}\ge 1/n$. On average, a large box $k$ is occupied because the (conditional) mean number of balls in it is $np_{k}\ge 1$. One may therefore expect that $K_{n}$ is asymptotically close to the number of large boxes, which is $\rho(n)$ (see \eqref{rho_def} for the definition). For the Bernoulli sieve, this heuristic was justified in \cite{GneIksMar:10} by showing that $K^{*}_{n}=K^{*}_{n}(1)$, properly centered and normalized, converges weakly iff $\rho^{*}(n)$, defined in \eqref{rho_def_BS} and centered and normalized by the same constants, converges weakly to the same law. From this, one may expect that $(K^{*}_{n}(t))_{t\in [0,1]}$ is well-approximated by $(\rho^{*}(n^{t}))_{t\in [0,1]}$, but this turns out
to be wrong. It will actually be shown in Section \ref{sec_approx} that the time-reversal
$$ \rho^{*}(n)-\rho^{*}(n^{(1-t)-})\ =\ \#\{k\in\N: n^{-1} < p^{*}_{k}\le n^{t-1}\},\quad t\in[0,1] $$ provides a uniform approximation for $(K^{*}_{n}(t))_{t\in [0,1]}$ which is tight enough to derive that $(K^{*}_{n}(t))_{t\in [0,1]}$, properly centered and normalized, converges weakly in the Skorohod space if the same is true for $(\rho^{*}(n)-\rho^{*}(n^{(1-t)-}))_{t\in [0,1]}$. This
new observation allows us to replace the existing methods based on Poissonization-de-Poissonization used in earlier works on the Bernoulli sieve and constitutes the first
principal contribution of the present paper. We stress that our Lemma \ref{approx_gen_bound} essentially shows that the behavior of $K_{n}(t):=\sum_{k=1}^{[n^{t}]}K_{n,r}$ for large $n$ is driven by that of $\rho(n)-\rho(n^{(1-t)-})$ for any Karlin occupancy scheme with deterministic or random $(p_{k})_{k\in\N}$ provided that $\rho(x)$ exhibits a
logarithmic growth and its increments satisfy an additional condition.

\vspace{.1cm}
Once a functional limit theorem for $(\rho^{*}(n^{t}))_{t\in [0,1]}$ has been proved, the corresponding functional limit theorem for $(\rho^{*}(n)-\rho^{*}(n^{(1-t)-}))_{t\in [0,1]}$
follows by an application of the continuous mapping theorem. Put
\begin{equation}\label{tstar_def}
T^{*}_{n}\ :=\ |\log W_{1}|+\ldots+|\log W_{n-1}|+|\log (1-W_{n})|,\quad n\in\N
\end{equation}
and note that $\rho^{*}(n^{t})=\#\{k\in\N: T^{*}_{k}\le t\log n\}$ equals the number of visits of the \emph{perturbed random walk} $(T^{*}_{n})_{n\in\N}$ to the interval $[0, t\log n]$. Theorem \ref{prw_flt} stated in Section \ref{sec_prw} provides several functional limit theorems for the number of visits of a general perturbed random walk, not necessarily related to the Bernoulli sieve. Being of independent interest, this result is the second principal contribution of the present paper.

\section{Perturbed random walks}\label{sec_prw}

Let $(\xi_{k},\eta_{k})_{k\in\N}$ be a sequence of independent copies of a $\R^{2}$-valued random vector $(\xi,\eta)$ with positive components. Set $S_{0}:=0$, $S_{n}:=\xi_{1}+\ldots+\xi_{n}$, $n\in\N$ and
then
\begin{equation*}\label{prw_def} T_{n}:=S_{n-1}+\eta_{n},\quad n\in\N.
\end{equation*}
The sequence $(T_{n})_{n\in\N}$ is called a \emph{perturbed random walk} and has recently attracted some interest in the literature, see \cite{AlsIksMei:15} and the references therein. Let $N(x)$ denote the number of visits of $(T_{n})_{n\in\N}$ to the interval
$[0,\,x]$, i.e.,
\begin{equation*}
N(x)\ :=\ \sum_{k\ge 1}\1_{\{T_{k}\le x\}}\ =\ \sum_{k\ge 0}\1_{\{S_{k}+\eta_{k+1}\le x\}},\quad x\ge 0.
\end{equation*}
We start with an assertion that will be used in the proof of Proposition \ref{probab}.

\begin{Prop}\label{strong}
If $\sfm:=\Erw\xi<\infty$, then
\begin{equation}\label{str}
\lim_{n\to\infty}\ \sup_{t\in [0,1]}\,\bigg|{\sfm(N(n)-N(n(1-t)-)\over n}-t\,\bigg|\ =\ 0\quad \text{a.s.}
\end{equation}
\end{Prop}

If $\Erw\eta<\infty$, the weak convergence of the finite-dimensional distributions of $(N(nt))_{t\ge 0}$, properly centered and normalized, follows from Theorem 2.4 in \cite{IksMarMei:16}, see also Example 3.2 there. Theorem \ref{prw_flt} given next is an extension of the aforementioned result to convergence in the Skorohod space. The standard approach to such a strengthening would be to prove tightness. Since this turned out beyond our reach we use an alternative approach.

\begin{Theorem}\label{prw_flt}
Let $T>0$ and $F(x)=\Prob\{\eta\le x\}$, $x\ge 0$. In (B1), (B2) and (B3) below, assume further that $\Erw \eta^{a}<\infty$ for some $a>0$.
\begin{description}[(A4)]\itemsep4pt
\item[(B1)] If $\sfs^2:={\rm Var}\,\xi<\infty$, then
\begin{equation}\label{prw_finite_variance}
\bigg(\frac{N(nt)-\sfm^{-1}\int_{0}^{nt}F(u)\,{\rm d}u}{\sqrt{\sfm^{-3}\sfs^2
n}}\bigg)_{t\in[0,T]}\ \overset{J_{1}}{\Longrightarrow}\ (B(t))_{t\in[0,T]}
\end{equation}
as $n\to\infty$, where $\sfm=\Erw \xi<\infty$ and $(B(t))_{t\in [0,T]}$ is a standard Brownian motion.
\item[(B2)] If $\sfs^2=\infty$ and
\begin{equation*}\label{prw_{n}ormalization_sv_sec_moment}
\Erw \xi^2 \1_{\{\xi\le x\}}\ \sim\ \ell(x),\quad x\to\infty
\end{equation*}
for some $\ell$ slowly varying at infinity, then
\begin{equation}\label{prw_clt_sv_sec_moment}
\bigg(\frac{N(nt)-\sfm^{-1}\int_{0}^{nt}F(u)\,{\rm d}u}{\sfm^{-3/2}c(n)}\bigg)_{t\in[0,T]}\ \overset{J_{1}}{\Longrightarrow}\ (B(t))_{t\in[0,T]}
\end{equation}
as $n\to\infty$, where $c$ is a positive function satisfying $\lim_{x\to\infty}c(x)^{-2}x\ell(c(x))=1$.
\item[(B3)] If
\begin{equation}\label{prw_{n}ormalization_rv_sec_moment}
\Prob\{\xi>x\}\ \sim\ x^{-\alpha}\ell(x),\quad x\to\infty
\end{equation}
for some $\alpha\in (1,2)$ and some $\ell$ slowly varying at infinity, then
\begin{equation}\label{prw_clt_rv_sec_moment}
\bigg(\frac{N(nt)-\sfm^{-1}\int_{0}^{nt}F(u)\,{\rm d}u}{\sfm^{-(\alpha+1)/\alpha}c(n)}\bigg)_{t\in[0,T]}\ \overset{M_{1}}{\Longrightarrow}\ 
(S_{\alpha}(t))_{t\in[0,T]}
\end{equation}
as $n\to\infty$, where $(S_{\alpha}(t))_{t\in [0,T]}$ is an $\alpha$-stable L\'{e}vy process, $S_{\alpha}(1)$ has characteristic function \eqref{cf_st}, and $c$ is a positive function satisfying $\lim_{x\to\infty}c(x)^{-\alpha} x\ell(c(x))=1$.
\item[(B4)] If \eqref{prw_{n}ormalization_rv_sec_moment} holds with $\alpha\in(0,1)$,
then
\begin{equation}\label{prw_clt_{i}nf_exp}
\bigg(\frac{\ell(n)N(nt)}{n^{\alpha}}\bigg)_{t\in[0,T]}\ \overset{J_{1}}{\Longrightarrow}\ (W^{\leftarrow}_{\alpha}(t))_{t\in[0,T]}
\end{equation}
as $n\to\infty$, where $W^{\leftarrow}_{\alpha}(t)$ is as defined in Theorem \ref{main2_theorem_fclt}.
\end{description}
\end{Theorem}

We close this section with two further results that will be needed in our analysis. The first one tells us that the maximal number of visits of a perturbed random walk to subintervals of $[0,n+b]$ of length $b$ grows stochastically more slowly than any positive power of $n$.

\begin{Prop}\label{prw_uniformity}
For all positive $b$ and $c$,
\begin{equation}\label{prw_to_zero_{i}n_probability}
\Plim_{n\to\infty}n^{-c} \sup_{t\in[0,1]}\big(N(nt+b)-N(nt)\big)\ =\ 0.
\end{equation}
\end{Prop}

The second result is a straightforward extension of the fact from renewal theory that the expected number of visits of a random walk with positive increments to intervals of length $y$ is bounded by a linear function in $y$.

\begin{Prop}\label{sub}
For all $x,y\ge 0$ and appropriate positive $C$ and $D$,
\begin{equation}\label{sub_con}
\Erw\big(N(x+y)-N(x)\big)\ \le\ Cy+D.
\end{equation}
\end{Prop}

\section{The Karlin occupancy scheme: an approximation result}\label{sec_approx}

In this section, we focus on the Karlin occupancy scheme with \emph{deterministic} $(p_{k})_{k\in\N}$. For $j=1,\ldots, n$, let $Z_{n,j}$ denote the number of balls in the $j^{th}$ box, so that
$$ K_{n}(t)\ :=\ \sum_{j\ge 1}\1_{\{Z_{n,j}\in [1,n^{t}]\}},\quad t\in [0,1], $$
gives the number of occupied boxes containing at most $[n^{t}]$ balls. Proposition \ref{approx_gen_bound} is the first main ingredient to the proof of Theorem \ref{main_theorem_fclt}. The connection with the Bernoulli sieve becomes clear
when conditioning on $(W_{k})$.
\begin{Prop}\label{approx_gen_bound}
Let $\rho(x)$ be as defined in \eqref{rho_def}. Then
\begin{align*}
&\Erw\sup_{t\in[0,1]}\left|K_{n}(t)-\bigg(\rho(n)-\rho\big(n^{(1-t)-}\big)\bigg)\right|\ 
\le\ 6\,\Big(\rho(n)-\rho\big(x_{0}^{-1}n(\log n)^{-2}\big)\Big)\nonumber\\
&\quad+\ \frac{3\rho(n)}{\log n}\ +\ \int_{1}^\infty t^{-2}(\rho(nt)-\rho(n))\ {\rm d}t\ +\ 2\sup_{t\in[0,1]}\Big(\rho(en^{1-t})-\rho(e^{-1}n^{1-t})\Big)\label{final!},
\end{align*}
where $x_{0}>1$ denotes an absolute constant that does not depend on $n$, nor on
$(p_{j})_{j\in\N}$.
\end{Prop}

\begin{proof}
Without further notice, all subsequent estimates, including $n^{t}-n^{3t/4}>1$ for $t\in[l_{n},\,1]$ with $l_{n}:=\frac{2\log\log n}{\log n}$, are meant under the proviso that $n$ be sufficiently large. We start with the basic inequality
\begin{align*}
\left|K_{n}(t)-\big(\rho(n)-\rho\big(n^{(1-t)-}\big)\big)\right|\ &\le\ \sum_{j\ge 1}\1_{\{Z_{n,j}>n^{t},\,np_{j}\in [1,n^{t}]\}}\ +\ \sum_{j\ge 1}\1_{\{Z_{n,j}=0,\,np_{j}\in [1,n^{t}]\}}\\
&+\ \sum_{j\ge 1}\1_{\{Z_{n,j}\in [1,n^{t}],\,np_{j} > n^{t}\}}\ +\ \sum_{j\ge 1}\1_{\{Z_{n,j}\in [1,n^{t}],\,np_{j}<1\}}\\
&=:\ S_{n}^{(1)}(t)+S_{n}^{(2)}(t)+S_{n}^{(3)}(t)+S_{n}^{(4)}(t)
\end{align*}
and intend to estimate $\Erw\sup_{t\in[0,1]}S_{n}^{(i)}(t)$ for
$i=1,2,3,4$.

\vspace{.1cm}
As for the first summand, we have
\begin{align*}
&\sup_{t\in[0,1]}\sum_{j\ge 1}\1_{\{Z_{n,j}>n^{t},\,np_{j}\in [1,n^{t}]\}}\\
&\hspace{1cm}\le\ \sup_{t\in[0,\,l_{n})}\sum_{j\ge 1}\1_{\{np_{j}\in [1,n^{t}]\}}\ +\ \sup_{t\in[l_{n},1]}\sum_{j\ge 1}\1_{\{Z_{n,j}>n^{t},\,np_{j}\in [1,n^{t}-n^{3t/4}]\}}\\
&\hspace{1.5cm}+\ \sup_{t\in[l_{n},\,1]}\sum_{j\ge 1}\1_{\{Z_{n,j}>n^{t},\,np_{j}\in (n^{t}-n^{3t/4},\,n^{t}]\}}\\
&\hspace{1cm}\le\ \sum_{j\ge 1}\1_{\{np_{j}\in [1,n^{l_{n}})\}}\ +\ \sup_{t\in[l_{n},\,1]}\sum_{j\ge 1}\1_{\{Z_{n,j}-np_{j}>n^{3t/4},\,np_{j}\in [1,n^{t}-n^{3t/4}]\}}\\
&\hspace{1.5cm}+\ \sup_{t\in[l_{n},\,1]}\sum_{j\ge 1}\1_{\{np_{j}\in (n^{t}-n^{3t/4},\,n^{t}]\}}\\
&\hspace{1cm}\le\ \sum_{j\ge 1}\1_{\{np_{j}\in [1,\,\log^2 n)\}}\ +\ \sum_{j\ge 1}\1_{\{Z_{n,j}-np_{j}>(np_{j})^{3/4},\,np_{j}\in [1,n]\}}\\
&\hspace{1.5cm}+\ \sup_{t\in[l_{n},\,1]}\sum_{j\ge 1}\1_{\{np_{j}\in (n^{t}-n^{3t/4},\,n^{t}]\}}\\
&=:\ S^{(11)}_{n}+S^{(12)}_{n}+S^{(13)}_{n}.
\end{align*}
By the definition of $\rho$,
\begin{equation*}
S^{(11)}_{n}\ =\ \rho(n)-\rho(n(\log n)^{-2}).
\end{equation*}
The random variable $Z_{n,j}$ has a binomial distribution with parameters $n$ and $p_{j}$, so that $\Erw Z_{n,j}=np_{j}$ and $\Var\,Z_{n,j}=np_{j}(1-p_{j})$. Use Chebyshev's inequality to obtain
\begin{align*}
\Erw S^{(12)}_{n}\ &\le\ \sum_{j\ge 1}\frac{np_{j}(1-p_{j})}{(np_{j})^{3/2}}\1_{\{np_{j}\in [1,n]\}}\ 
\le\ \sum_{j\ge 1}(np_{j})^{-1/2}\1_{\{np_{j}\in [1,n]\}}\\
&=\ \int_{[1,n/(\log n)^2]} n^{-1/2}x^{1/2}\ {\rm d}\rho(x)\ +\ \int_{(n/(\log n)^2,\, n]}n^{-1/2}x^{1/2}\ {\rm d}\rho(x)\\
&\le\ \frac{\rho(n)}{\log n}\ +\ \Big(\rho(n)-\rho\big(n(\log
n)^{-2}\big)\Big).
\end{align*}
Finally, the third term $S_{n}^{(13)}$ can be estimated as follows:
\begin{align*}
S_{n}^{(13)}\ &=\ \sup_{t\in[l_{n},\,1]}\sum_{j\ge 1}\1_{\{p_{j}\in (n^{t-1}-n^{3t/4-1},\,n^{t-1}]\}}\\&\le\ \sup_{t\in[l_{n},1]}\sum_{j\ge 1}\1_{\{p_{j}\in (n^{t-1}(1-n^{-l_{n}/4}),\,n^{t-1}]\}}\\
&\le\ \sup_{t\in[0,1]}\sum_{j\ge 1}\1_{\{p_{j}\in (n^{t-1}(1-n^{-l_{n}/4}),\,n^{t-1}]\}}\\
&=\ \sup_{t\in[0,1]}\sum_{j\ge 1}\1_{\{p_{j}\in (n^{t-1}(1-(\log n)^{-1/2}),\,n^{t-1}]\}}\\
&\le\ \sup_{t\in[0,1]}\Big(\rho\big(en^{1-t}\big)-\rho\big(n^{1-t}\big)\Big).
\end{align*}
Summarizing,
\begin{align*}
\Erw\sup_{t\in[0,1]}S_{n}^{(1)}(t)\ &\le\ 2\Big(\rho(n)-\rho\big(n(\log n)^{-2}\big)\Big)\ 
+\frac{\rho(n)}{\log n}\ +\ \sup_{t\in[0,1]}\Big(\rho\big(en^{1-t}\big)-\rho\big(n^{1-t}\big)\Big).
\end{align*}
Since $S^{(2)}_{n}(t)$ and $S^{(4)}_{n}(t)$ are monotone in $t$, we further infer more easily that
\begin{align*}
\Erw \sup_{t\in[0,1]}S^{(2)}_{n}(t)\ &=\ \sum_{j\ge 1}\Prob\{Z_{n,j}=0\}\1_{\{np_{j}\in [1,n]\}}\ =\ 
\sum_{j\ge 1}(1-p_{j})^n\1_{\{np_{j}\in [1,n]\}}\\
&\le\ \sum_{j\ge 1}e^{-p_{j}n}\1_{\{1/n\le p_{j} \le 1\}}\\
&=\ \int_{[1,n/\log n]}e^{-n/x}{\rm d}\rho(x)+\int_{(n/\log n,\,n]}e^{-n/x}{\rm d}\rho(x)\\
&\le\ \frac{\rho(n)}{n}\ +\ \Big(\rho(n)-\rho(n\big(\log n)^{-1}\big)\Big)\\
&\le\ \frac{\rho(n)}{\log n}\ +\ \Big(\rho(n)-\rho\big(n(\log n)^{-2}\big)\Big)
\end{align*}
and, with the help of Markov's inequality,
\begin{align*}
\Erw \sup_{t\in[0,1]}S^{(4)}_{n}(t)\ &=\ \sum_{j\ge 1}\Prob\{Z_{n,j}\ge 1\}\1_{\{np_{j} \le 1\}}\ \le\  \sum_{j\ge 1}np_{j}\1_{\{np_{j}\le 1\}}\\
&=\ \int_{[n,\,\infty)}n x^{-1}\ {\rm d}\rho(x)\ =\ \int_{[1,\,\infty)}x^{-1}\ {\rm d}(\rho(nx)-\rho(n)).
\end{align*}
Upon integration by parts and a use of the fact that $\lim_{x\to\infty}x^{-1}\rho(x)=0$ (see
Lemma 3 in \cite{Karlin:67}), we arrive at
\begin{align*}
\Erw \sup_{t\in[0,1]}S^{(4)}_{n}(t)\ &\le\ \int_{1}^\infty x^{-2}(\rho(nx)-\rho(n))\ {\rm d}x\ +\ \Big(\rho(n)-\rho(n-)\Big)\\
&\le\ \int_{1}^{\infty} x^{-2}(\rho(nx)-\rho(n))\ {\rm d}x\ +\ \Big(\rho(n)-\rho\big(n(\log n)^{-2}\big)\Big).
\end{align*}
Left with $S^{(3)}_{n}(t)$, we write 
\begin{align*}
\sup_{t\in[0,1]}&\sum_{j\ge 1}\1_{\{Z_{n,j}\in [1,n^{t}],\,np_{j} > n^{t}\}}\\
&\le\ \sup_{t\in[0,1]}\sum_{j\ge 1}\1_{\{Z_{n,j}\le n^{t},\,np_{j}-(np_{j})^{3/4}\ge n^{t}\}}\ +\ 
\sup_{t\in[0,1]}\sum_{j\ge 1}\1_{\{np_{j}-(np_{j})^{3/4} < n^{t}\le np_{j}\}}\\
&\le\ \sum_{j\ge 1}\1_{\{Z_{n,j}\le np_{j}-(np_{j})^{3/4},\,np_{j}\ge 1\}}\ +\ \sup_{t\in[0,1]}\sum_{j\ge 1}\1_{\{np_{j}-(np_{j})^{3/4}<n^{t}\le np_{j}\}}\\
&=\ \sum_{j\ge 1}\1_{\{n-Z_{n,j}-n(1-p_{j})\ge
(np_{j})^{3/4},\,np_{j}\in [1,n]\}}\ +\ \sup_{t\in[0,1]}\sum_{j\ge 1}\1_{\{np_{j}-(np_{j})^{3/4} < n^{t}\le np_{j}\}}.
\end{align*}
Since $n-Z_{n,j}$ has a binomial distribution with parameters $n$ and $1-p_{j}$, Chebyshev's inequality along with the estimates used for $\Erw S_{n}^{(12)}$ yields
$$ \Erw \sum_{j\ge 1}\1_{\{n-Z_{n,j}-n(1-p_{j})\ge (np_{j})^{3/4},\,np_{j}\in [1,n]\}}\ \le\ \frac{\rho(n)}{\log n}\ +\ \Big(\rho(n)-\rho\big(n(\log n)^{-2}\big)\Big).
$$
To find a proper bound for the second summand, we first verify that
$$ x-x^{3/4} < y \le x\quad\text{and}\quad y\ge 1 $$
entail
\begin{equation}\label{inverse_{i}nequality}
y\le x < y+(x_{0}-1)y^{3/4},
\end{equation}
where $x_{0}>1$ is the unique positive solution to the equation
$x-x^{3/4}=1$. Indeed, if $x<x_{0}$, then
$$x\,=\,(x-x^{3/4})+x^{3/4}<y+x_{0}^{3/4}\,=\,y+(x_{0}-1)\,\le\, y+(x_{0}-1)y^{3/4},$$ 
where the last inequality is a consequence of $y\ge 1$.
On the other hand, if $x\ge x_{0}$ we have
\begin{align*}
x\,&=\,(x-x^{3/4})+(1-x^{-1/4})^{-3/4}(x-x^{3/4})^{3/4}\\
&<\,y+(1-x_{0}^{-1/4})^{-3/4}(x-x^{3/4})^{3/4}\\
&<\,y+(1-x_{0}^{-1/4})^{-3/4}y^{3/4}\,=\,y+(x_{0}-1)y^{3/4},
\end{align*}
where the first inequality follows from the fact that $x\mapsto(1-x^{-1/4})^{-3/4}$ is nonincreasing on $(1,\infty)$, giving
$(1-x^{-1/4})^{-3/4}\le (1-x_{0}^{-1/4})^{-3/4}$ for $x\ge x_{0}$. In view of \eqref{inverse_{i}nequality},
\begin{align*}
\sup_{t\in [0,1]}&\sum_{j\ge 1}\1_{\{np_{j}-(np_{j})^{3/4} < n^{t}\le np_{j}\}}\ \le\ \sup_{t\in[0,1]}\sum_{j\ge 1}\1_{\{n^{t}\le np_{j} < n^{t}+(x_{0}-1)n^{3t/4}\}}\\
&\le\ \sup_{t\in[0,\,l_{n}]}\sum_{j\ge 1}\1_{\{n^{t}\le np_{j} < n^{t}+(x_{0}-1)n^{3t/4}\}}\ +\ \sup_{t\in[l_{n},\,1]}\sum_{j\ge 1}\1_{\{n^{t}\le np_{j} < n^{t}+(x_{0}-1)n^{3t/4}\}}\\
&\le\ \sum_{j\ge 1}\1_{\{1\le np_{j} < x_{0}n^{l_{n}}\}}+\sup_{t\in[l_{n},\,1]}\sum_{j\ge 1}\1_{\{p_{j}\in[n^{t-1},\,n^{t-1}(1+(x_{0}-1)n^{-t/4}))\}}\\
&\le\ \Big(\rho(n)-\rho\big(x_{0}^{-1}n(\log n)^{-2}\big)\Big)\ +\ \sup_{t\in[l_{n},\,1]}\sum_{j\ge 1}\1_{\{p_{j}\in[n^{t-1},\,n^{t-1}(1+(x_{0}-1)n^{-l_{n}/4}))\}}\\
&\le\ \Big(\rho(n)-\rho\big(x_{0}^{-1}n(\log n)^{-2}\big)\Big)\ +\ \sup_{t\in[0,1]}\sum_{j\ge 1}\1_{\{p_{j}\in[n^{t-1},\,n^{t-1}(1+(x_{0}-1)(\log n)^{-1/2}))\}}\\
&\le\ \Big(\rho(n)-\rho\big(x_{0}^{-1}n(\log n)^{-2}\big)\Big)\ +\ \sup_{t\in [0,1]}\Big(\rho\big(n^{1-t}\big)-\rho\big(e^{-1}n^{1-t}\big)\Big),
\end{align*}
where $l_{n}=2\log \log n/ \log n$ should be recalled. A combination of the previous estimates completes the proof of the proposition.\qed
\end{proof}

\section{Proofs for Section \ref{sec_prw}}\label{sec_proofs}

\begin{proof}[of Proposition \ref{strong}]
If we can prove that
\begin{align}\label{inter}
&\lim_{n\to\infty}\frac{N(ns)}{n}\ =\ \lim_{n\to\infty}\frac{N(ns-)}{n}\ =\ \frac{s}{\sfm}\quad\text{a.s.},
\shortintertext{for any $s>0$, then}
&\lim_{n\to\infty}\frac{\sfm\,\big(N(n)-N(n(1-t)-)\big)}{n}\ =\ t\quad\text{a.s.}\nonumber
\end{align}
for all
$t\in [0,1]$, and this yields \eqref{str} because, by Dini's theorem, convergence of monotone functions to a continuous limit is uniform on compact sets.

\vspace{.1cm}
\noindent {\sc Proof of \eqref{inter}}. Since $N(ns)-N(ns-)\leq 1$ for all $n\in\N$ and $s>0$, it suffices to consider $N(ns)$. Setting
\begin{equation}\label{nut}
\nu(t):=\inf\{k\in\N:S_{k}>t\}=\sum_{k\ge 0}\1_{\{S_{k}\le t\}},\quad t\in\R,
\end{equation}
we use the following estimate
\begin{equation}\label{eq:N(ns)/n estimate}
\frac{\nu(ns-y)}{n}\ -\ \frac{1}{n}\sum_{k=1}^{\nu(ns)}\1_{\{\eta_{k}>y\}}\ \le\ \frac{N(ns)}{n}\ \le\ \frac{\nu(ns)}{n}\quad\text{a.s.}
\end{equation} 
valid for any $y>0$ and $n$ sufficiently large. The strong law of large numbers provides us with $\lim_{n\to\infty}n^{-1}\sum_{k=1}^n\1_{\{\eta_{k}>y\}}=\Prob\{\eta>y\}$, while the same
law for renewal counting processes \cite[Theorem 5.1 on p.~57]{Gut:09} gives
$\lim_{n\to\infty}n^{-1}\nu(ns)=\sfm^{-1}s$ a.s. Consequently,
$\lim_{n\to\infty}n^{-1}\sum_{k=1}^{\nu(ns)}\1_{\{\eta_{k}>y\}}=\sfm^{-1}s\Prob\{\eta>y\}$ a.s. Finally, \eqref{inter} follows from \eqref{eq:N(ns)/n estimate} by first letting $n$ and then $y$ tend to infinity.\qed
\end{proof}

\begin{proof}[of Theorem \ref{prw_flt}]
(B4) is covered by Corollary 2.6 in \cite{IksMarMei:16}, see also Example 3.2 there.

\vspace{.1cm}\noindent 
\textsc{Proof of (B1)-(B3)}. By Theorem 1.1 in \cite{Iksanov:13b} applied to $h(t)=F(t)=\Prob\{\eta\le t\}$, we know that relations
\eqref{prw_finite_variance}, \eqref{prw_clt_sv_sec_moment} and
\eqref{prw_clt_rv_sec_moment} hold true with $\sum_{k\ge 0}F(nt-S_{k})\1_{\{S_{k}\le nt\}}$ in place of $N(nt)$. Put 
$$ X(t)\ :=\ \sum_{k\ge 0}\left(\1_{\{S_{k}+\eta_{k+1}\le t\}}-F(t-S_{k})\1_{\{S_{k}\le t\}}\right) $$ 
for $t\ge 0$. In view of the representation 
$$N(t)\ -\ \frac{1}{\sfm}\int_{0}^{t} F(u)\ {\rm d}u\ =\ X(t)\ +\ \left(\sum_{k\ge 0}F(t-S_{k})\1_{\{S_{k}\le t\}}-\frac{1}{\sfm}\int_{0}^{t} F(u){\rm d}u\right) $$ 
and Slutsky's lemma, it suffices to check that
\begin{equation}\label{prw_first_summand_{n}egligible}
\Plim_{n\to\infty}n^{-1/2}\sup_{0\le t\le T}|X(nt)|\ =\ 0,
\end{equation}
using also $\lim_{n\to\infty}n^{-1/2}c(n)=\infty$ in the situation of (B2) and (B3) (see Lemma \ref{cn} in the Appendix).

\vspace{.1cm}
Suppose we can prove that
\begin{equation}\label{prw_to_zero_as}
\lim_{t\to\infty}t^{-1/2}X(t)\ =\ 0\quad\text{a.s.}
\end{equation}
Then, by using
\begin{align*}
n^{-1/2}\sup_{0\le t\le T}|X(nt)|\ &\le\ n^{-1/2}\sup_{0\le t\le s}|X(t)|\ +\ n^{-1/2}\sup_{s\le t\le nT}|X(t)|\\
&\le\ n^{-1/2}\sup_{0\le t\le s}|X(t)|\ +\ T^{1/2}\sup_{t\ge s}|t^{-1/2}X(t)|
\end{align*}
for $0<s<nT$ and sending first $n$ and then $s$ to infinity, we see that \eqref{prw_to_zero_as} implies \eqref{prw_first_summand_{n}egligible}.

\vspace{.1cm}
Passing to the proof of \eqref{prw_to_zero_as}, we first observe that for each
$t\ge 0$, there exists $m\in\N_{0}$ such that $t\in [m,m+1)$ and
\begin{align*}
t^{-1/2}X(t)\ &\le\ m^{-1/2}\sum_{k\ge 0}\big(\1_{\{S_{k}+\eta_{k+1}\le m+1\}}-F(m+1-S_{k})\1_{\{S_{k}\le m+1\}}\big)\\
&+\ m^{-1/2}\sum_{k\ge 0}\big(F(m+1-S_{k})\1_{\{S_{k}\le m+1\}}-F(m-S_{k})\1_{\{S_{k}\le m\}}\big).
\end{align*}
Obviously, a lower estimate of similar kind holds as well so that \eqref{prw_to_zero_as} is a consequence of the two limit assertions
\begin{align}\label{prw_to_zero_as11}
&\hspace{3.5cm}\lim_{\N\ni m\to\infty} m^{-1/2}X(m)\ =\ 0\quad\text{a.s.}
\shortintertext{and}
\label{prw_to_zero_as12}
&\lim_{\N\ni m\to\infty} m^{-\delta}\sum_{k\ge 0}\big(F(m+1-S_{k})\1_{\{S_{k}\le
m+1\}}-F(m-S_{k})\1_{\{S_{k}\le m\}}\big)\ =\ 0\quad\text{a.s.}
\end{align}
for any $\delta>0$.\footnote{At this point it is enough to prove
\eqref{prw_to_zero_as12} for $\delta=1/2$ only. However, we shall
need \eqref{prw_to_zero_as12} for arbitrary $\delta>0$ later on.}

\vspace{.1cm}\noindent
\textsc{Proof of \eqref{prw_to_zero_as11}}. We start by noting that $X(t)$ equals the terminal value of the martingale $(R(k,t),\mathcal{F}_{k})_{k\in\N_{0}}$, where $R(0,t):=0$,
$\mathcal{F}_{0}:=\{\Omega,\oslash\}$ denotes the trivial
$\sigma$-algebra,
$$R(k,t):=\sum_{j=0}^{k-1}\big(\1_{\{S_{j}+\eta_{j+1}\le t\}}-F(t-S_{j})\1_{\{S_{j}\le t\}}\big)$$ and $\mathcal{F}_{k}:=\sigma((\xi_{j}, \eta_{j}): 1\le j\le k)$ for $k\in\N$. For any
$l\in\N$, use the Burkholder-Davis-Gundy inequality \cite[Theorem 11.3.2]{Chow+Teicher:97} to obtain
\begin{align*}
&\Erw (X(t))^{2l}\\
&\le\ C\left(\Erw\left[\sum_{k\ge 0}\Erw\big((R(k+1,t)-R(k,t))^2|\mathcal{F}_{k}\big)\right]^{l}+\sum_{k\ge 0}\Erw\big(R(k+1,t)-R(k,t)\big)^{2l}\right)\\
&=\ C\left(\Erw\left[\sum_{k\ge 0}F(t-S_{k})(1-F(t-S_{k}))\1_{\{S_{k}\le t\}}\right]^{l}\right.\\
&\hspace{1cm}+\ \left.\sum_{k\ge 0}\Erw \big(\1_{\{S_{k}+\eta_{k+1}\le
t\}}-F(t-S_{k})\1_{\{S_{k}\le t\}}\big)^{2l}\right)\ =:\ C(I_{1}(t)+I_2(t))
\end{align*}
for a positive constant $C$.

\vspace{.1cm}
Since $\sum_{k\ge 0}F(t-S_{k})(1-F(t-S_{k}))\1_{\{S_{k}\le t\}}\le\sum_{k\ge 0}(1-F(t-S_{k}))\1_{\{S_{k}\le t\}}$, we find
\begin{equation}\label{2}
I_{1}(t)=O\Bigg(\bigg(\int_{0}^{t}(1-F(y)){\rm d}y\bigg)^{l}\,\Bigg),\quad
t\to\infty
\end{equation}
when applying Lemma \ref{impo1} to the nonincreasing function $t\mapsto 1-F(t)$. Furthermore,
\begin{align*}
\Erw\bigg(&\big(\1_{\{S_{k}+\eta_{k+1}\le t\}}-F(t-S_{k})\1_{\{S_{k}\le t\}}\big)^{2l}\bigg|\mathcal{F}_{k}\bigg)\\
&=\ \big(F(t-S_{k})(1-F(t-S_{k}))^{2l}+(1-F(t-S_{k}))(F(t-S_{k}))^{2l}\big)\1_{\{S_{k}\le
t\}}\\
&\le\ (1-F(t-S_{k}))\1_{\{S_{k}\le t\}},
\end{align*}
giving $I_2(t)\le \Erw \sum_{k\ge 0}(1-F(t-S_{k}))\1_{\{S_{k}\le t\}}$ and thereupon
$$ I_2(t)\ =\ O\bigg(\int_{0}^{t}(1-F(y)){\rm d}y\bigg),\quad t\to\infty $$ 
by Lemma \ref{impo1} in the Appendix. We have thus shown that
\begin{equation}\label{prw_to_zero_as3}
\Erw X(t)^{2l}=O\left(\bigg(\int_{0}^{t}(1-F(y))\ {\rm d}y\bigg)^{l}\,\right),\quad t\to\infty,
\end{equation}
and so $\Erw X(t)^{2l}$ is of order $O(1)$ in the case $\Erw\eta=\int_{0}^\infty (1-F(y)){\rm d}y<\infty$. If $\Erw\eta=\infty$, then our assumption $\Erw\eta^{a}<\infty$ for some $a>0$ entails $a\in (0,1)$. Using \eqref{prw_to_zero_as3} in combination with $\lim_{t\to\infty}t^a(1-F(t))=0$, clearly a consequence of $\Erw \eta^a<\infty$, yields $\Erw X(m)^{2l}=O(m^{l(1-a)})$ as $m\to\infty$. Hence, for all $\eps>0$,
$$ \Prob\{|X(m)|>\eps m^{1/2}\}\ \le\ \frac{\Erw X(m)^{2l}}{\eps^{l}m^{l}}\ =\ O(m^{-la}),\quad m\to\infty, $$
by Markov's inequality. Choosing $l:=\min\{j\in\N: ja\ge 2\}$ in the last estimate yields \eqref{prw_to_zero_as11} by the Borel-Cantelli lemma.

\vspace{.1cm}
\noindent {\sc Proof of \eqref{prw_to_zero_as12}}. Set
$J(m):=\int_{[0,\,m]}\big(F(m+1-y)-F(m-y)\big)\,{\rm d}\nu(y)$ for
$m\in\N$, where $\nu(\cdot)$ is defined by \eqref{nut}. We will use the following estimate
\begin{align*}
0\ &\le\ \sum_{k\ge 0}\big(F(m+1-S_{k})\1_{\{S_{k}\le m+1\}}\ -\ F(m-S_{k})\1_{\{S_{k}\le
m\}}\big)\\
&=\ \int_{[0,\,m+1]}F(m+1-y)\ {\rm d}\nu(y)\ -\ \int_{[0,\,m]}F(m-y)\ {\rm d}\nu(y)\\
&\le\ J(m)+\nu(m+1)-\nu(m).
\end{align*}
Since $\lim_{m\to\infty}m^{-\delta}(\nu(m+1)-\nu(m))=0$ a.s.\ by Lemma \ref{impo2} in the Appendix, it remains to examine $J(m)$. But
\begin{align*}
J(m)\ &=\ F(m+1)-F(m)\ +\ \sum_{k=0}^{m-1}\int_{(k,\,k+1]}\big(F(m+1-y)-F(m-y)\big){\rm
d}\nu(y)\\
&\le\ 1\ +\ \sum_{k=0}^{m-1}\big(F(m+1-k)-F(m-1-k)\big)(\nu(k+1)-\nu(k))\\
&\le\ 1\ +\ (F(m)+F(m+1)-F(1))\max_{0\le k\le m-1}(\nu(k+1)-\nu(k)),
\end{align*}
implies $\lim_{m\to\infty}m^{-\delta}J(m)=0$ a.s.\ by
Lemma \ref{impo2}, and this completes the proof of Theorem \ref{prw_flt}.\qed
\end{proof}

\begin{proof}[of Proposition \ref{prw_uniformity}]
Here it is obviously enough to prove that
\begin{equation}\label{prw_uniformity_to_zero_as12}
\lim_{t\to\infty} t^{-c}(N(t+b)-N(t))\ =\ 0\quad\text{a.s.}
\end{equation}
for $b>0$, w.l.o.g.\ $b=1$ and $t\to\infty$ along integers only. Put
$$ Y(t)\ :=\ \sum_{j\ge 0}\1_{\{t<S_{j}+\eta_{j+1}\le t+1\}}-\sum_{j\ge 0}(F(t+1-S_{j})\1_{\{S_{j}\le t+1\}}-F(t-S_{j})\1_{\{S_{j}\le t\}}) $$
In view of \eqref{prw_to_zero_as12}, relation \eqref{prw_uniformity_to_zero_as12} follows if we can show that
$$
\lim_{n\to\infty} n^{-c}Y(n)=0\quad\text{a.s.}
$$
which in turn follows from
\begin{equation}\label{prw_uniformity_bounded_moments}
\Erw Y(t)^{2l}=O(1),\quad t\to\infty
\end{equation}
for every $l\in\N$ by a similar argument as in the previous proof using Markov's inequality and the Borel-Cantelli lemma.

\vspace{.1cm}
Left with \eqref{prw_uniformity_bounded_moments}, the subsequent argument is very similar to the corresponding one for $\Erw X(t)^{2l}$ in the previous proof. Again, $Y(t)$ is the terminal value of a martingale with respect to the filtration $(\mathcal{F}_{k})_{k\in\N_{0}}$ from there, viz.
$$ Y(k,t):=\sum_{j=0}^{k-1}\1_{\{t<S_{j}+\eta_{j+1}\le t+1\}}-\sum_{j=0}^{k-1}(F(t+1-S_{j})\1_{\{S_{j}\le t+1\}}-F(t-S_{j})\1_{\{S_{j}\le t\}}). $$
By another use of the Burkholder-Davis-Gundy inequality, one finds that
\begin{align*}
\Erw (Y(t))^{2l}\ &\le\ C\left(\Erw\Bigg(\sum_{k\ge 0}p(t-S_{k})(1-p(t-S_{k}))\1_{\{S_{k}\le t\}}\Bigg)^{l}+1\right)\\
&\le\ C\left(\Erw\Bigg(\sum_{k\ge 0}p(t-S_{k})\1_{\{S_{k}\le t\}}\Bigg)^{l}+1\right)
\end{align*}
for a positive constant $C$ and with $p(x):=F(x+1)-F(x)$. By Lemma \ref{impo1},
\begin{equation*}
\Erw\left(\sum_{k\ge 0}p(t-S_{k})\1_{\{S_{k}\le t\}}\right)^{l}\ =\ 
O\left(\sum_{j=0}^{[t]}\sup_{y\in[j,\,j+1)}(F(y+1)-F(y))\right)^{l}\ =\ O(1)
\end{equation*}
because
\begin{eqnarray*}
\Bigg(\sum_{j=0}^{[t]}\sup_{y\in[j,\,j+1)}(F(y+1)-F(y))\Bigg)^{l}\ \le\ 
\Bigg(\sum_{j=0}^{[t]}(F(j+2)-F(j))\Bigg)^{l}\ \le\ 2^{l}.
\end{eqnarray*}
This establishes \eqref{prw_uniformity_bounded_moments} thereby
finishing the proof of Lemma \ref{prw_uniformity}.\qed
\end{proof}

\begin{proof}[of Proposition \ref{sub}]
Note that
$$ U(x)\ :=\ \Erw\nu(x)\ =\ \sum_{j\ge 1}\Prob\{S_{j-1}\le x\},\quad x\in\R, $$
is the renewal function associated with $(S_{n})_{n\in\N}$. 
As a consequence of the distributional subadditivity of $\nu$ (see formula (5.7) on p.~58 in \cite{Gut:09}), the monotonicity of $\nu$, and the fact that $\nu(x)=0$ for $x<0$, $U$ is subadditive on $\R$, i.e. $U(x+y)\le U(x)+U(y)$ for all $x,y\in\R$, and so
\begin{equation*}\label{U_lin_bound}
U(x)\le Cx^{+}+C_2
\end{equation*}
for all $x\in\R$ and some positive $C$ and $D$. By using these facts, we finally obtain
\begin{align*}
\Erw\big(N(x+y)-N(x)\big)\ &=\ \sum_{j\ge 1}\Prob\{x<T_{j}\le x+y\}\\
&=\ \sum_{j\ge 1}\int_{[0,\,\infty)}\Prob\{x-z< S_{j-1}\le x+y-z\}\ \Prob\{\eta\in {\rm d}z\}\\
&=\ \int_{[0,\,\infty)}\big(U(x+y-z)-U(x-z)\big)\ \Prob\{\eta\in {\rm d}z\}\\
&\le\ U(y)\ \le\ Cy+D.
\end{align*}
for all $x,y\ge 0$.\qed
\end{proof}

\section{Proofs for Section \ref{main}}

Recall that the Bernoulli sieve is the Karlin occupancy scheme with the random probabilities $(p^{*}_{k})_{k\in\N}$ defined in \eqref{BS_frequencies}. Condition on $(p^{*}_{k})_{k\in\N}$ and apply Proposition \ref{approx_gen_bound} to obtain
\begin{align}
\Erw&\Bigg(\sup_{t\in [0,1]}|K^{*}_{n}(t)-(\rho^{*}(n)-\rho^{*}(n^{(1-t)-}))|\bigg|(p^{*}_{j})\Bigg)\nonumber\\
\begin{split}\label{final_BS}
&\le\ 6\Big(\rho^{*}(n)-\rho^{*}\big(x_{0}^{-1}n(\log n)^{-2}\big)\Big)\ +\ \frac{3\rho^{*}(n)}{\log n}\ +\ \int_{1}^\infty\frac{\rho^{*}(nx)-\rho^{*}(n)}{x^{2}}\ {\rm d}x\\
&\hspace{1cm}+\ 2\sup_{t\in[0,1]}\Big(\rho^{*}(en^{1-t})-\rho^{*}(e^{-1}n^{1-t})\Big)\ =:\ \eps_{n}
\end{split}
\end{align}
with $\rho^{*}(x)$ defined in \eqref{rho_def_BS}. The next lemma
shows that $\eps_{n}$ does not grow faster in probability than any power of $\log n$.

\begin{Lemma}\label{bs_approx_lemma}
For all $c>0$,
$$ \Plim_{n\to\infty}\frac{\eps_{n}}{(\log n)^{c}}\ =\ 0. $$
\end{Lemma}

\begin{proof}
Observe that
\begin{equation}\label{connect}
\rho^{*}(x)=N(\log x)
\end{equation}
for $N(\cdot)$ corresponding to $\xi=|\log W|$ and $\eta=|\log(1-W)|$. Using \eqref{sub_con}, we infer that the expectation of the first three terms of \eqref{final_BS} is $O(\log\log n)$. Hence the sum of these terms divided by $(\log n)^c$ converges to zero in
probability by Markov's inequality. Finally, the fourth
term
\begin{equation*}
\sup_{t\in [0,1]}\big(\rho^{*}(en^{1-t})-\rho^{*}(e^{-1}n^{1-t})\big)\ =\ \sup_{t\in [0,1]}\big(N(t\log n+1)-N(t\log n-1)\big)
\end{equation*}
is of order $o((\log n)^c)$ in probability for any $c>0$ by Lemma \ref{prw_uniformity}.\qed
\end{proof}

\begin{proof}[of Proposition \ref{probab}]
By Lemma \ref{bs_approx_lemma},
$$ \Plim_{n\to\infty}\frac{1}{\log n}\Erw\Bigg(\sup_{t\in
[0,1]}|K^{*}_{n}(t)-(\rho^{*}(n)-\rho^{*}(n^{(1-t)-}))|\Big|(p^{*}_{j})\Bigg)\ =\ 0 $$ 
which implies
\begin{equation}\label{pr1}
\Plim_{n\to\infty}\frac{1}{\log n}\sup_{t\in [0,1]}\big|K^{*}_{n}(t)-(\rho^{*}(n)-\rho^{*}(n^{(1-t)-}))\big|\ =\ 0
\end{equation}
by Markov's inequality and the dominated convergence theorem. Furthermore,
$$ \bigg|\frac{K_{n}^{*}}{\mu^{-1}\log n}-1\bigg|\ \le\ \frac{|K_{n}^{*}-\rho^{*}(n)|}{\mu^{-1}\log n}\ +\ \bigg|\frac{\rho^{*}(n)}{\mu^{-1}\log n}-1\bigg|\quad\text{a.s.} $$ 
The first term on the right-hand side converges to zero in probability by \eqref{pr1} (recall that $K_{n}^{*}=K_{n}^{*}(1)$), and the second does so by \eqref{connect} and Proposition \ref{strong}. Hence,\footnote{Though not needed here, let us note that the convergence in \eqref{kn} holds a.s. Just apply the Borel-Cantelli lemma to $K^{*}_{[\exp(n^2)]}$ and then use the monotonicity of $K^{*}_{n}$ in $n$. Alternatively, this follows from Theorems 1' and 8 in \cite{Karlin:67} in combination with $\lim_{n\to\infty}(\log n)^{-1}\rho^{*}(n)=\mu^{-1}$ a.s.\ provided by Proposition
\ref{strong}.}
\begin{equation}\label{kn}
\Plim_{n\to\infty}\frac{K_{n}^{*}}{\log n}\ =\ \frac{1}{\mu}.
\end{equation}
To complete the proof, we observe that
\begin{align*}
\sup_{t\in [0,1]}\bigg|{K_{n}^{*}(t)\over K_{n}^{*}}-t\bigg|\ &=\ {\mu^{-1}\log n\over
K_{n}^{*}}\left({\sup_{t\in [0,1]}|K_{n}^{*}(t)-tK_{n}^{*}|\over \mu^{-1}\log n}\right)\\
&\le\ {\mu^{-1}\log n\over K_{n}^{*}}\Bigg({\sup_{t\in [0,1]}|K_{n}^{*}(t)-\big(\rho^{*}(n)-\rho^{*}(n^{(1-t)-})\big)|\over\mu^{-1}\log n}\\
&\hspace{.8cm}+\ \sup_{t\in [0,1]}\bigg|{\mu(\rho^{*}(n)-\rho^{*}(n^{(1-t)-}))\over \log
n}-t\bigg|+\bigg|{K_{n}^{*}\over \mu^{-1}\log n}-1\bigg|\Bigg)
\end{align*}
and note that the right-hand side converges to zero in probability by \eqref{pr1}, \eqref{kn} and Proposition \ref{strong}.\qed
\end{proof}

\begin{proof}[of Theorem \ref{main_theorem_fclt}]

\vspace{.1cm}\noindent
{\sc Proof for $K^{*}_{n}(t)$}. Denote by $a_{n}$ the normalization used for $K^{*}_{n}(t)$ in the respective parts of Theorem \ref{main_theorem_fclt}. Observe that $a_{n}$ grows faster than some power of the logarithm. Hence, by \eqref{final_BS}, Lemma \ref{bs_approx_lemma} and Markov's inequality,
$$ \Plim_{n\to\infty}\Prob\left\{\sup_{t\in [0,1]}|K^{*}_{n}(t)-(\rho^{*}(n)-\rho^{*}(n^{(1-t)-}))|>\eps a_{n}\bigg|(p^{*}_{j})\right\}\ =\ 0 $$ 
for all $\epsilon>0$. Using the dominated convergence theorem, this yields
$$ \Plim_{n\to\infty}a_{n}^{-1}\sup_{t\in [0,1]}|K^{*}_{n}(t)-(\rho^{*}(n)-\rho^{*}(n^{(1-t)-}))|\ =\ 0. $$
In view of representation \eqref{repr}, it remains to prove Theorem \ref{main_theorem_fclt} with $\rho^{*}(n)-\rho^{*}(n^{(1-t)-})$ replacing $K^{*}_{n}(t)$. Using \eqref{connect}, we see that this is accomplished by an application of Theorem \ref{prw_flt} to the process
$(\rho^{*}(n^{t}))_{t\in [0,1]}$ 
and the subsequent use of the continuous mapping theorem. For the
latter, three supporting facts are:
\begin{description}[(3)]\itemsep3pt
\item[(1)] in the $J_{1}$- ($M_{1}$-)
topology on $D[0,1]\times D[0,1]$
$$ \bigg({\rho^{*}(n)-u_{n}(1)\over a_{n}},\,
{-\rho^{*}(n^{(1-t)-})+u_{n}(1)-u_{n}(t)\over a_{n}}\bigg)\ \Longrightarrow\ \big(X(1), -X((1-t)-)\big), $$ 
as $n\to\infty$, where $X$ is a Brownian motion and the convergence is in the $J_{1}$-topology in (C1) and (C2), whereas $X$ is an $\alpha$-stable L\'{e}vy process and the convergence is in the $M_{1}$-topology in (C3); 
\item[(2)] if $D[0,1]\times D[0,1]$ is equipped with the $J_{1}$ or $M_{1}$-topology (which is stronger than the product topology), then the mapping $\psi: D[0,1]\times D[0,1]\to D[0,1]$, defined by
$\psi(x,y):=x+y$ is continuous;  
\item[(3)] $(Y(1)-Y((1-t)-))_{t\in [0,1]}\eqdist (Y(t))_{t\in [0,1]}$ for any L\'{e}vy process $Y$.
\end{description}

\vspace{.2cm}\noindent {\sc Proof for $K^{*}_{n}(t)/K^{*}_{n}$}. Write
\begin{align}
{\log n\over \mu a_{n}}\left(\frac{K^{*}_{n}(t)}{K^{*}_{n}}-t+\frac{v_{n}(t)-tv_{n}(1)}{u_{n}(1)}\right)\ 
&=\ \frac{\log n}{\mu K^{*}_{n}}\left(\frac{K^{*}_{n}(t)-tK^{*}_{n}+v_{n}(t)-tv_{n}(1)}{a_{n}}\right)\notag\\
&\hspace{.8cm}+\ \frac{\log n}{\mu K^{*}_{n}}\frac{v_{n}(t)-tv_{n}(1)}{u_{n}(1)}{K_{n}^{*}-u_{n}(1)\over a_{n}}\label{imp}
\end{align}
with $a_{n}$ as before. By what has already been proved, we know that
\begin{equation}\label{cor_proof1}
\left(\frac{K^{*}_{n}(t)-u_{n}(t)}{a_{n}}\right)_{t\in[0,1]}\ \Longrightarrow\ (Z(t))_{t\in[0,1]},\quad n\to\infty
\end{equation}
with some L\'{e}vy process $Z$ depending on the respective case. Use the continuous
mapping theorem along with continuity of the summation mapping and \eqref{kn} to obtain
$$ {\log n\over \mu K^{*}_{n}}\left(\frac{K^{*}_{n}(t)-tK^{*}_{n}+v_{n}(t)-tv_{n}(1)}{a_{n}}\right)\ \Longrightarrow\ (Z(t)-tZ(1))_{t\in[0,1]},\quad n\to\infty. $$ 
In view of \eqref{kn}, \eqref{cor_proof1} and
$$ \sup_{t\in[0,1]}\left|\frac{v_{n}(t)-tv_{n}(1)}{u_{n}(1)}\right|\ \le\ \frac{2v_{n}(1)}{u_{n}(1)}\ \to\ 0,\quad n\to\infty, $$
the second term in \eqref{imp} converges to zero in probability uniformly in
$t\in [0,1]$. The proof is completed by an appeal to Slutsky's lemma.\qed
\end{proof}

\begin{proof}[of Theorem \ref{main2_theorem_fclt}]
As for $K_{n}^{*}(t)$, the argument used in the proof of Theorem \ref{main_theorem_fclt} applies here without changes, and \eqref{clt_{i}nf_exp} follows. Passing to the proof for
$K_{n}^{*}(t)/K^{*}_{n}$, we immediately conclude that \eqref{clt_{i}nf_exp} entails
$$ \bigg(\frac{\ell(\log n)K^{*}_{n}(t)}{(\log n)^{\alpha}},\,{(\log n)^\alpha\over\ell(\log n)K_{n}^{*}}\bigg)\ \Longrightarrow\ \bigg(W^{\leftarrow}_{\alpha}(1)-W^{\leftarrow}_{\alpha}((1-t)-),\,{1\over W^{\leftarrow}_\alpha(1)}\bigg),\quad n\to\infty $$
in the
$J_{1}$-topology on $D[0,1]\times D[0,1]$. To complete the proof, use once again the continuous mapping theorem together with the following fact: if $D[0,1]\times D[0,1]$ is equipped with the $J_{1}$-topology, then the mapping $\phi: D[0,1]\times D[0,1]\to
D[0,1]$ defined by $\phi(x,y):=xy$ is continuous.\qed
\end{proof}

\section{Appendix}

We collect three auxiliary results that have been used in the proofs of the main results.

\begin{Lemma}\label{cn}
Let $c(x)$ be a function as defined in (B2) or (B3) of Theorem \ref{prw_flt}. Then
$$ \lim_{x\to\infty}x^{-1/2}c(x)\ =\ \infty. $$
\end{Lemma}
\begin{proof}
For (B3), this is immediate because
$\lim_{x\to\infty}c^{-\alpha}(x)x\ell(c(x))=1$ implies
that $c$ varies regularly with index $1/\alpha>1/2$. Suppose now the assumptions of part (B2) be valid. Since $c(x)$ is the asymptotic inverse of $x\mapsto x^2/\ell(x)$ and
$\lim_{x\to\infty} x^2/\ell(x)=\infty$, we infer that $\lim_{x\to\infty}c(x)=\infty$. Moreover,
$\lim_{x\to\infty}\ell(x)=\infty$ in view of $\sfs^2=\infty$. Thus,
$\lim_{x\to\infty}\ell(c(x))=\infty$ which in
combination with $x^{-1}c^2(x)\sim \ell(c(x))$, $x\to\infty$,
entails $\lim_{x\to\infty}x^{-1/2}c(x)=\infty$.\qed
\end{proof}

Let $(S_{n})_{n\in\N_{0}}$ be a zero-delayed standard random walk with
positive steps as in Section \ref{sec_prw}. Recall the notation
$\nu(t)=\inf\{k\in\N: S_{k}>t\}$ for $t\in\R$. Plainly, $\nu(t)\equiv 0$ for $t\le 0$.

\begin{Lemma}\label{impo2}
For all $\delta>0$,
\begin{equation}\label{impo2-6} 
\lim_{\N\ni n\to\infty}\, n^{-\delta} \max_{0\le j\le n}(\nu(j+1)-\nu(j))\ =\ 0\quad\text{a.s.}
\end{equation}
\end{Lemma}
\begin{proof}
Fix $a>0$. Since $\lim_{\gamma\to\infty}\Erw e^{-\gamma S_{1}}=\Prob\{S_{1}=0\}=0$, we can pick $\gamma>0$ such that $e^a\Erw e^{-\gamma S_{1}}<1$. Then
\begin{align}\label{7}
e^a{\Erw e^{a\nu(1)}-1\over e^a-1}\ &=\ \sum_{j\ge 1}e^{aj}\Prob\{\nu(1)\ge j\}\ =\ \sum_{j\ge
1}e^{aj}\Prob\{S_{j-1}\le 1\}\notag\\
&\le\ e^\gamma \sum_{j\ge 1}e^{aj}\big(\Erw e^{-\gamma S_{1}}\big)^{j-1}\ <\ \infty
\end{align}
by Markov's inequality. Hence, for all $\eps>0$,
$$ \Prob\{\nu(k+1)-\nu(k)>\eps k^\delta\}\ \le\ \Prob\{\nu(1)>\eps k^\delta\}\ \le\ e^{-a\eps
k^\delta}\Erw e^{a\nu(1)} $$
having utilized the distributional subadditivity of $\nu(t)$ \cite[formula (5.7) on p.~58]{Gut:09}) for the first inequality and Markov's inequality
for the second. By the Borel-Cantelli lemma, we conclude that $\lim_{n\to\infty}n^{-\delta}(\nu(n+1)-\nu(n))=0$ a.s.\ and thereupon \eqref{impo2-6}.\qed
\end{proof}

\begin{Lemma}\label{impo1}
Let $G:[0,\infty)\to [0,\infty)$ be a locally bounded function, $l\in\N$ be an arbitrary integer, and $[t]$ denote the integer-part of $t$. Then, as $t\to\infty$,
\begin{equation}\label{impo1-gen}
\Erw\left(\sum_{k\ge 0}G(t-S_{k})\1_{\{S_{k}\le t\}}\right)^{l}\ =\ O\left(\Bigg(\sum_{j=0}^{[t]}\sup_{y\in[j,\,j+1)}G(y)\Bigg)^{l}\,\right).
\end{equation}
If $G$ is nonincreasing, then
\begin{equation}\label{impo1-mon}
\Erw\left(\sum_{k\ge 0}G(t-S_{k})\1_{\{S_{k}\le t\}}\right)^{l}\ =\ O\left(\bigg(\int_{0}^{t} G(y){\rm d}y\bigg)^{l}\,\right).
\end{equation}
as $t\to\infty$.
\end{Lemma}
\begin{proof}
The result is trivial if $G\equiv 0$. Assuming that $G$ is not identically zero, we have
$$ G(t)\ \le\ \sum_{n=0}^{[t]}\bigg(\sup_{y\in[n,\,n+1)}G(y)\bigg)\1_{[n,\,n+1)}(t),\quad t\ge 0. $$
Hence, for $t$ so large that the right-hand side is positive, we infer
\begin{align}
\begin{split}\label{impo1-5}
&\Bigg(\sum_{k\ge 0}(G(t-S_{k}))\1_{\{S_{k}\le t\}}\Bigg)^{l}\ \le\ 
\left(\sum_{n=0}^{[t]}\left(\sup_{y\in[n,n+1)}G(y)\right)(\nu(t-n)-\nu(t-(n+1)))\right)^{l}\\
&\quad\le\ \Bigg(\sum_{k=0}^{[t]}\sup_{y\in[k,\,k+1)}G(y)\Bigg)^{l}\sum_{n=0}^{[t]}{\sup_{y\in[n,\,n+1)}G(y)\over\sum_{k=0}^{[t]}\sup_{y\in[k,\,k+1)}G(y)}\,(\nu(t-n)-\nu(t-(n+1)))^{l},
\end{split}
\end{align}
having utilized $\sum_{j\ge 0}\1_{\{a<S_{j}\le b\}}=\nu(b)-\nu(a)$ a.s.\ for $a<b$ and the convexity of $x\mapsto x^{l}$. It is clear that
$$\nu (t-[t])-\nu(t-[t]-1)\ =\ \nu(t-[t])\ \le\ \nu(1)\quad\text{a.s.} $$ 
Moreover, $\Prob\{\nu(t-n)-\nu(t-(n+1))>x\}\le \Prob\{\nu(1)>x\}$ for $n=0,\ldots, [t]-1$ and $x\ge 0$ by the distributional subadditivity of $\nu(t)$ mentioned earlier. Consequently,
\begin{equation}\label{impo1-4}
\Erw (\nu(t-n)-\nu(t-(n+1)))^{l}\ \le\ \Erw (\nu(1))^{l}\ <\ \infty\quad\text{for}\quad n=0,\ldots, [t],
\end{equation}
where the finiteness follows from \eqref{7}.

\vspace{.1cm}
Passing to expectations in \eqref{impo1-5} and using \eqref{impo1-4}, we see that
\begin{align*}
\Erw\Bigg(\sum_{j\ge 0}G(t-S_{j})\1_{\{S_{j}\le t\}}\Bigg)^{l}\ \le\ 
\Bigg(\sum_{n=0}^{[t]}\sup_{y\in[n,\,n+1)}G(y)\Bigg)^{l}\,\Erw\nu(1)^{l}
\end{align*}
which proves \eqref{impo1-gen}. If $G$ is nonincreasing, then
$$ \sum_{n=0}^{[t]}\sup_{y\in[n,\,n+1)}G(y)\ =\ \sum_{n=0}^{[t]}G(n)\ =\ \int_{0}^{t}G(y){\rm
d}y\ +\ O(1),\quad t\to\infty, $$ 
and \eqref{impo1-mon} follows.\qed
\end{proof}

\vspace{1cm}
\footnotesize
\noindent   {\bf Acknowledgements.}
The research of Gerold Alsmeyer was supported by the Deutsche Forschungsgemeinschaft (SFB 878), the research of Alexander Marynych by the Alexander von Humboldt Foundation. This work was finished while Alexander Iksanov was visiting the Institute of Mathematical Statistics at M\"unster. He gratefully acknowledges financial support and hospitality.

\def\cprime{$'$}


\end{document}